\documentclass[12pt,reqno]{amsart}

\usepackage{blindtext}
\usepackage{geometry}
 \usepackage{mathtools}

\usepackage{times}
\usepackage[T1]{fontenc}
\usepackage{mathrsfs}
\usepackage{latexsym}
\usepackage[dvips]{graphics}
\usepackage{epsfig}
\usepackage{blindtext}
\usepackage{microtype}
\usepackage{breqn}
\usepackage[breaklinks]{hyperref}
\usepackage{spverbatim}
\usepackage{xcolor}
\usepackage{physics}
\usepackage{tikz}
\usepackage{listings}
\usepackage{amsmath,amsfonts,amsthm,amssymb,amscd}
\input amssym.def
\input amssym.tex
\usepackage{color}

\newcommand\be{\begin{equation}}
\newcommand\ee{\end{equation}}
\newcommand\bea{\begin{eqnarray}}
\newcommand\eea{\end{eqnarray}}
\newcommand\bi{\begin{itemize}}
\newcommand\ei{\end{itemize}}
\newcommand\ben{\begin{enumerate}}
\newcommand\een{\end{enumerate}}

\newtheorem{thm}{Theorem}[section]
\newtheorem{conj}[thm]{Conjecture}
\newtheorem{cor}[thm]{Corollary}
\newtheorem{lem}[thm]{Lemma}

\theoremstyle{definition}
\newtheorem{defi}[thm]{Definition}
\newtheorem{rek}[thm]{Remark}
\theoremstyle{thm}

\newcommand{\Vol}{\mathrm{Vol}}
\newcommand{\Prob}{\mathbb{P}}
\usepackage{bbm}

\newcommand{\R}{\ensuremath{\mathbb{R}}}

\newcommand{\Z}{\ensuremath{\mathbb{Z}}}

\numberwithin{equation}{section}

\begin{document}

\title{Benfordness of measurements resulting from box fragmentation}

\author{Livia Betti}
\email{lbetti@u.rochester.edu} \address{Department of Mathematics, University of Rochester, Rochester, NY 14627}
\author{Irfan Durmi\'{c}}
\email{idurmic@student.jyu.fi} \address{Department of
  Mathematics and Statistics, University of Jyväskylä, Jyväskylä, FI, 40740}
\author{Zoe McDonald}
\email{zmcd@bu.edu} \address{Department of Mathematics and Statistics, Boston University, Boston, MA, 02215}
\author{Jack B. Miller}
\email{jack.miller.jbm82@yale.edu} \address{Department of Mathematics, Yale University, New Haven, CT, 06511}
\author{Steven J. Miller}
\email{Steven.J.Miller@williams.edu} \address{Department of
  Mathematics and Statistics, Williams College, Williamstown, MA 01267}
 


\thanks{This work was supported by NSF grant DMS1947438, and Williams College. }

\begin{abstract}
We make progress on a conjecture made by \cite{DM}, which states that the $d$-dimensional frames of $m$-dimensional boxes resulting from a fragmentation process satisfy Benford's law for all $1 \leq d \leq  m$.
We provide a sufficient condition for Benford's law to be satisfied, namely that the maximum product of $d$ sides is itself a Benford random variable.
Motivated to produce an example of such a fragmentation process, we show that processes constructed from log-uniform proportion cuts satisfy the maximum criterion for $d=1$.
\end{abstract}

\date{\today}

\maketitle

\tableofcontents


\section{Introduction}\label{sec: introduction}

\subsection{Historical Background}
At the dawn of the $20^{th}$ century, the astronomer and mathematician Simon Newcomb 
observed that the logarithmic books at his workplace showed a lot of wear and tear at the 
early pages, but the more he progressed through the book, the less usage could be observed.
Newcomb deduced that his colleagues had a "bias" towards numbers starting with the digit 
$1$. In particular, the digit $1$ shows up as the first digit roughly $30\%$ of the time, 
the digit $2$ about $17\%$ of the time, and so on. While he did come up with a mathematical 
model for this interesting relationship, his work stayed mostly unnoticed.

It took another 57 years after Newcomb's discovery for physicist Frank Benford to make the exact same observation as Newcomb: the first pages of logarithmic tables were used far more than others. He formulated this law as follows. 
\begin{defi}\cite[Page 554]{Ben} \label{def:benforiginal}
We say that data exhibits \emph{(weak) Benford behavior} if the frequency $F_d$ of leading digit $d$ satisfies
\bea
  F_d \ = \ \log_{10}{\frac{d \ + \ 1}{d}}.\label{eqn:benorg}
\eea
\end{defi}

Nowadays, Benford's Law is used in detecting many different forms of fraud, and its prevalence in the world fascinates not only mathematicians, but many other scientists as well (to learn more about Benford's Law and its many applications, we recommend~\cite{BeHi, Nig, Mil1} to name a few). 




In 1986, Lemons \cite{Lemons} proposed using Benford's law to analyze the partitioning of a conserved quantity. Since then, driven by the potential application to nuclear fragmentation, mathematicians and physicists have taken an interest in the Benfordness of various fragmentation processes. Among these processes of interest is \textit{stick fragmentation}.
In the unrestricted stick fragmentation model, one begins with a stick of length $L$. Draw $p_1$ from a probability distribution on $(0,1)$. This fragments the stick into two sub-sticks of lengths $p_1L$ and $(1-p_1)L$. For each sub-stick, draw another independent probability ($p_2$ and $p_3$, respectively) from the same distribution. Repeat this process $N$ times.
Of particular interest is whether this fragmentation process follows Benford's law.

\subsection{Previous Work on Fragmentation}

An important definition when studying a more precise statistical version of Benford's law is the notion of the significand of a real number, i.e., its leading digits in scientific notation.

\begin{defi}[Significand]
Given a positive real number $x$, we say that its \textit{significand base $B>1$}, denoted $S_B(x)$, is the unique real number $S_B(x) \in [1,B)$ such that $k = \log_B(x)-\log_B(S_B(x))$ is an integer. One can then write $x = S_B(x) \cdot B^{k}$.
\end{defi}

As is common practice with these techniques involving proofs of Benford's law, we define a stricter version of Benford behavior.

\begin{defi}[Strong Benford's Law]
We say that a sequence of random variables $X^{(n)}$ \textit{converges to strong Benford behavior} in the base $B$ if
\begin{equation}
\label{eq:StrongBenfordBehavior}
\Prob(S_B(X^{(n)}) \leq D) \to \log_B(D) ,
\end{equation}
for all $D \in [1,B]$. Notice by compactness that this implies uniform convergence of (\ref{eq:StrongBenfordBehavior}).
\end{defi}

We may now state the previous results on box fragmentation.
Becker, et al. \cite{Beck} proved a theorem regarding unrestricted stick fragmentation (compare with their Theorem 1.5) which was later generalized by \cite{DM} in the form of the following theorem.

\begin{thm}[Benfordness of the $m$-Volumes of a Branching-Fragmentation Process]
\label{thm:Irfan'sResult}
Fix a continuous probability density $f:(0,1)\to\R$ such that its Mellin transform\footnote{The Mellin transform is related to the Fourier transform by a logarithmic change of variables, which we will discuss further in Section 4. Often, the Mellin and Fourier transforms are useful tool for stating regularity conditions.} $\mathcal{M}[f_u]$ satisfies
\begin{equation}
\label{eq:MellinCondition}
\lim_{n\to\infty} \sum_{\substack{\ell = -\infty \\ \ell \neq 0}}^\infty \left| {\prod_{u=1}^{nm} \mathcal{M}[f_u]\left(1-\frac{2\pi i \ell}{\log 10}\right)} \right| \ = \ 0,
\end{equation}
where each $f_u(t)$ is either $f(t)$ or $f(1-t)$ (the density of $1-P$ if $P$ has 
density $f$).  Given an $m$-dimensional box of $m$-dimensional volume $V$, 
we independently choose density cuts $p_1, p_2, \dots,p_{nm \ - \ 1}, 
p_{nm}$ from the unit interval stemming from the probability density function $f$ and the associated random variable $P$. 
After $N$ iterations we have
\begin{align}
V_1 \ &= \  Vp_1p_2p_4\cdots p_{2^{nm -2}} p_{2^{nm - 1}}, \quad V_2 \ = \ Vp_1p_2p_4\cdots p_{2^{nm -2}}(1-p_{2^{nm - 1}}), \quad \ldots,
\nonumber \\
V_{(2^m)^n} \ &= \    V(1-p_1)(1-p_3)(1-p_7) \cdots (1-p_{2^{nm - 1}-1})  (1-p_{2^{nm}-1}).
\end{align}
Let $\varphi_s$ denote the significand indicator function
\begin{equation}
    \varphi_s (x) \ \coloneqq \ \begin{cases}
        1 \ \ \ s_{10}(x) \ \leq \ s \\
        0 \ \ \ \mbox{otherwise}
    \end{cases}.
\end{equation}
Let $\rho_n(s)$ denote the fraction of volumes $V_1,\ldots,V_{(2^m)^n}$ with significand at most $s$, i.e.,
\begin{equation}
\rho_m^{(n)}(s) \ := \  \frac{ \sum_{i = 1}^{(2^m)^n} \varphi_s(V_i)}{(2^m)^n}.
\end{equation}
We have that the following two conditions hold.
\begin{enumerate}
\item $\lim _{n \to \infty} \mathbb{E} [\rho_m^{(n)}(s)] = \log_{10}(s)$,
\item $\lim _{n \to \infty} {\rm Var}\left(\rho_m^{(n)}(s)\right) = 0$.
\end{enumerate}
Thus, in the limit, the $m$-dimensional volumes resulting from such a branching-fragmentation process exhibit Benford behavior with high probability.
\end{thm}

\begin{figure}[h]
\begin{center}
\scalebox{.7}{\includegraphics{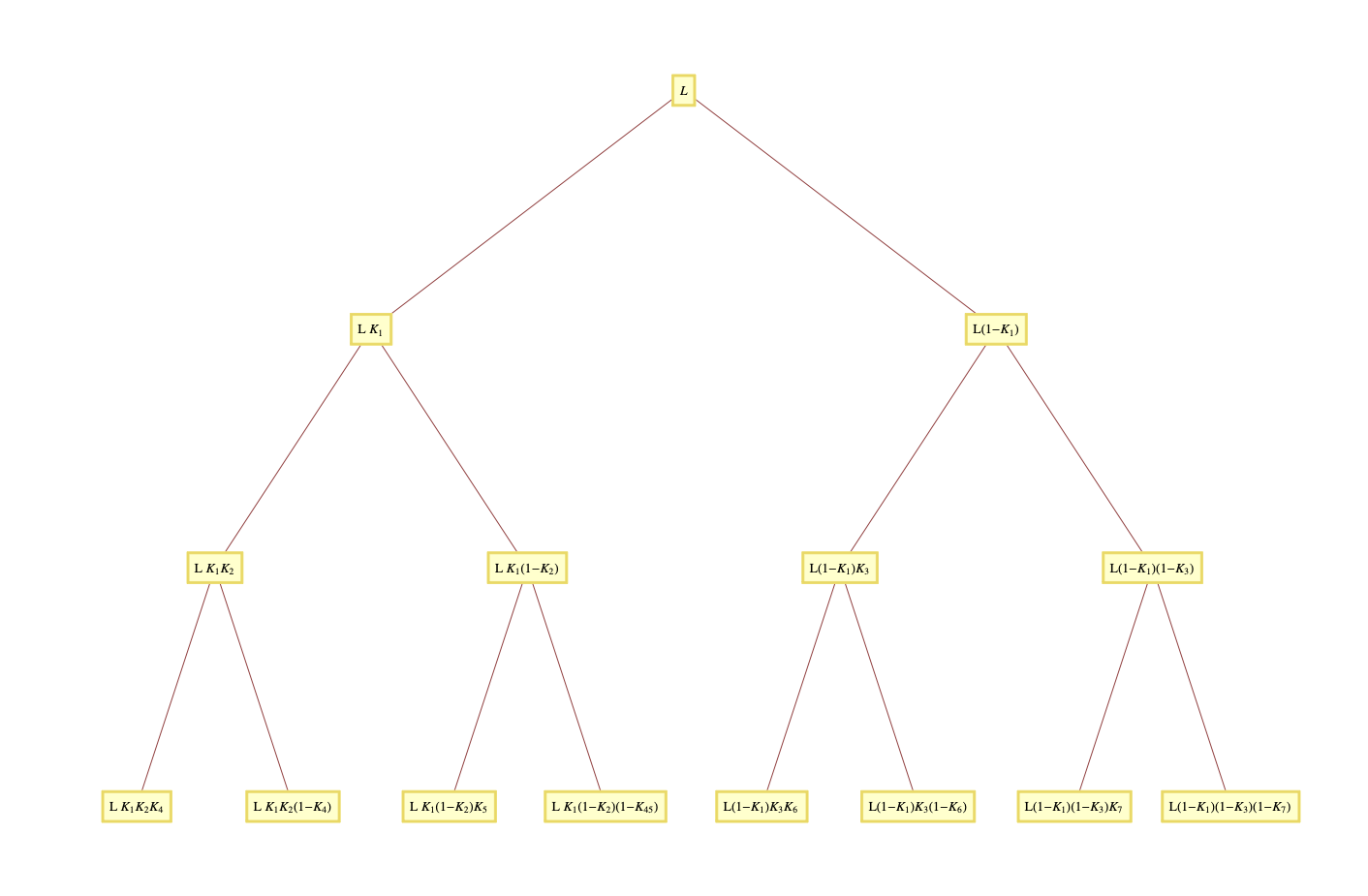}}
\caption{\label{fig:tree3levels} The side lengths of a one-dimensional branching-fragmentation  process for $n = 3$.}
\end{center}\end{figure}

\begin{figure}[h]
    \centering
    \scalebox{1.2}{\includegraphics{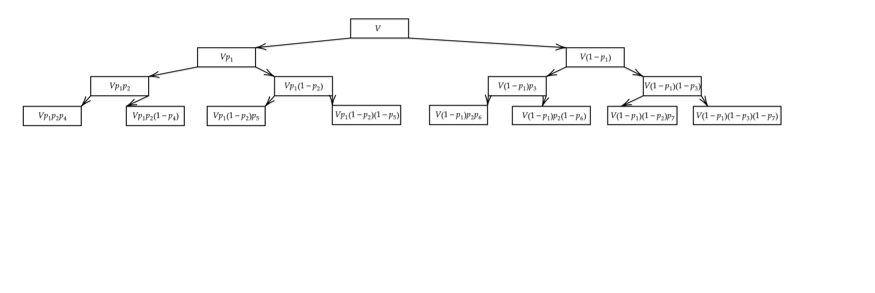}}
    \vspace{-40mm}
    \caption{The volumes of a three-dimensional branching-fragmentation process for $n = 1$.}
    \label{fig:3dfragmentationunrestricted}
\end{figure}

\begin{rek}
The exact fragmentation process used in Theorem \ref{thm:Irfan'sResult} features $2^{m \cdot n}$ boxes at time step $n$, all of which are concurrent sub-boxes of the original box. We say that this is a \textit{branching-fragmentation process}, as there are exponentially many boxes which naturally are the leaves of a height $n$ binary tree of all the boxes at all the time steps up to $n$.
Theorem \ref{thm:Irfan'sResult} proves strong concentration, i.e., that the variance goes to zero; morally this is because early decisions in the tree about where to cut have little effect on future boxes that are far apart leaves on the tree.
\end{rek}

The proof of Theorem \ref{thm:Irfan'sResult} suggests that one might observe Benford behavior in the perimeter, area, and other generalized volumes of lower-dimensional faces of boxes resulting from fragmentation.

\subsection{Results}

We prove results about linear-fragmentation processes, which we define as follows.

\begin{figure}[h]
\centering
{\includegraphics[width=15cm]{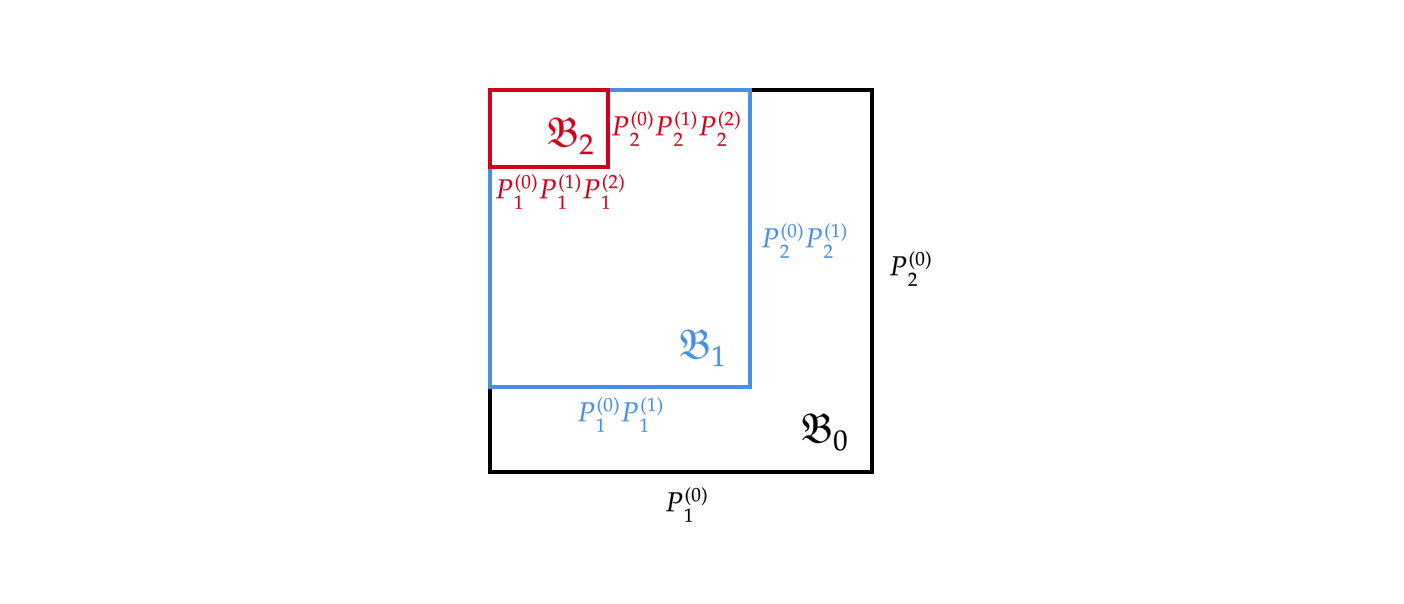}\vspace{-10mm}
    \caption{A linear fragmentation process for $n = 2$ on a two-dimensional box.}}
    \label{fig:linearfragmenatation}
\end{figure}

\begin{defi}[Box]
We say a set $\mathfrak{B} \subset \R^m$ is an \textit{$m$-dimensional box} if it is a set of the form $[a_1,b_1] \times \dotsi \times [a_m,b_m] \subset \R^m$, where $a_i<b_i$ are finite numbers.
\end{defi}

\begin{defi}[Linear-Fragmentation Process]
\label{def:LFP}
A \textit{linear-fragmentation process} is a sequence of random variables $\mathfrak{B}_0, \mathfrak{B}_1, \mathfrak{B}_2, \ldots$ such that the following hold.
\begin{enumerate}
    \item The random variables $\mathfrak{B}_i$ are $m$-dimensional boxes.
    \item The random variables $\mathfrak{B}_i$ form a descending chain $\mathfrak{B}_0 \supset \mathfrak{B}_1 \supset \mathfrak{B}_2 \supset \dotsi$.
    \item The distribution of $\mathfrak{B}_{n+1}$ conditioned on $\mathfrak{B}_n$ is some fixed distribution of independent proportion cuts $P_1,\ldots,P_m$ along each Cartesian axis. These $P_i$ are fixed over all $n\geq0$.
    \item The proportion cuts $P_i$ are continuous random variables with finite mean, variance, and third moment.
    \item\label{item:same mean and variance} We assume for simplicity of analysis that $\mathbb{E}[\log_B P_i]$ and ${\rm Var}[\log_B P_i]$ are constants $\mu_P\in\R$ and $\sigma_P^2 > 0$ that are uniform over $1 \leq i \leq m$.
\end{enumerate}
\end{defi}

The statistics we are interested in studying are the volumes of the frame random variables in a linear-fragmentation process.

\begin{defi}[$d$-Volume]
Given an $m$-dimensional box $\mathfrak{B}$ and a positive integer $d \leq m$, we say the \textit{$d$-volume} of $\mathfrak{B} = \prod_i [a_i,b_i]$ is the sum of the $d$-dimensional volumes of the $d$-dimensional faces of $\mathfrak{B}$. More precisely, we define
\begin{equation}
\label{def:dvolume}
\Vol_d(\mathfrak{B})
\ \coloneqq \
2^{m-d}
\sum_{|I| \ = \ d} \
\prod_{i \ \in \ I} (b_i-a_i) ,
\end{equation}
where we are summing over all subsets $I\subset\{1,\ldots,m\}$ with cardinality $d$.
\end{defi}

In Section \ref{sec:Redux2Max}, we prove the following theorem.

\begin{thm}[Maximum Criterion]
\label{thm:Redux2Max}
Let $\mathfrak{B} = \mathfrak{B}_0$ be a fixed $m$-dimensional box. 
Let $\mathfrak{B}_0 \supset \mathfrak{B}_1 \supset \dotsi$ be a a linear-fragmentation process whose proportion cuts $P_i$ have probability density functions $f_i:(0,1)\to(0,\infty)$.
Let
\begin{equation}
V_d^{(n)} \ \coloneqq \ \Vol_d(\mathfrak{B}_n)
\end{equation}
be the sequence of volumes obtained from this process. Let $\mathfrak{m}_d^{(n)}$ denote the maximum product of $d$ sides at each stage. If $\mathfrak{m}_d^{(n)}$ converges to strong Benford behavior, then so too does $V_d^{(n)}$ converge to strong Benford behavior as $n\to\infty$.
\end{thm}

\begin{rek}
Condition (\ref{item:same mean and variance}) for a linear-fragmentation process can be dropped with more work. The idea is that, by the law of large numbers, one expects the significand of our volumes to be largely influenced by the sides whose proportion cuts have the largest mean; therefore we have a reduction to the case of equal means. Having the same mean and different variances, there is little quantitative difference in our analysis, but for sake of notation it is much clearer to assume that all variances are the same.
\end{rek}

When $d = m$ there is only one choice of product, and therefore the maximum criterion is automatically satisfied by a large class of continuous proportion distributions, namely all such distributions $P_i$ for which repeated independent multiplications by $X = P_1 \dotsi P_m$ converges to strong Benford behavior. Note that this gives us a result analogous to those of \cite{Beck} and \cite{DM} for the linear-fragmentation process. Therefore, Theorem \ref{thm:Redux2Max} implies the following corollary.

\begin{cor}[Benfordness of the $m$-Volumes of a Linear-Fragmentation Process]
\label{cor:Benfordness of m-Volume Linear-Fragmentation}
Let $\mathfrak{B} = \mathfrak{B}_0$ be a fixed $m$-dimensional box. 
Let $\mathfrak{B}_0 \supset \mathfrak{B}_1 \supset \dotsi$ be a a linear-fragmentation process. Then the sequences of box volumes $\Vol_m(\mathfrak{B}_n)$ converges to strong Benford behavior.
\end{cor}

Indeed, one may appeal to the Central Limit Theorem in order to immediately see that the independent products of such $P_i$ satisfy the strong version of Benford's law.

In Section \ref{sec:ExampleFamily}, we produce an example family of distributions which satisfy the maximum criterion for $d=1$, namely those for which $\log_B P_i$ are uniformly distributed. We prove the following theorem.

\begin{thm}[Example of the Maximum Criterion being Satisfied]
\label{thm:SpecialCaseMaxSideLength}
Let $P_i^{(j)}$ be IID log-uniform distributions. In the case of $d=1$, i.e., perimeter, the maximum side-lengths
\begin{equation}
\mathfrak{m}_1^{(n)} \ \coloneqq \ 
\max_{1 \leq i \leq m} \ P_i^{(1)} \dotsi P_i^{(n)}
\end{equation}
converge to Strong Benford behavior as $n\to\infty$.
\end{thm}

In view of Theorem \ref{thm:Redux2Max}, this gives an example of Benford behavior for lower dimensional volumes of a box fragmentation process.

\begin{cor}
\label{cor:BenfordPerimeter}
Let $\mathfrak{B} = \mathfrak{B}_0$ be a fixed $m$-dimensional box. Let $\mathfrak{B}_0 \supset \mathfrak{B}_1 \supset \dots$ be a linear-fragmentation process whose proportion cuts $P_i$ are identically log-uniform. Then the sequence of frame perimeters $\Vol_1(\mathfrak{B}_n)$ converges to Strong Benford behavior as $n\to\infty$.
\end{cor}

\section{Reduction to the Maximum-Volume Face}
\label{sec:Redux2Max}

In this section, the following notation is fixed. We work under the assumptions of Definition \ref{def:LFP}. \vspace{5pt}\\
$\bullet$\quad $B$: a fixed base in $[1,\infty)$. \\
$\bullet$\quad $m$: the dimension of the boxes $\mathfrak{B} = \mathfrak{B}_0 \supset \mathfrak{B}_1 \supset \dotsi$. \\
$\bullet$\quad $d$: the dimension of the frames we are considering. \\
$\bullet$\quad $P_1^{(0)},\ldots,P_m^{(0)}$: the initial side lengths (i.e., $b_i-a_i$) of $\mathfrak{B}_0$. \\
$\bullet$\quad $P_1^{(n)},\ldots,P_m^{(n)}$, $n\geq 1$: the $m$ proportions drawn at the $n$th iteration. \\
$\bullet$\quad $S_i^{(n)}
\ \coloneqq \
\prod_{t=0}^n P_i^{(t)}$: the side lengths of $\mathfrak{B}_n$. \\
$\bullet$\quad $I,J$: dummy indexing sets ranging over subsets of $\{1,\ldots,m\}$ with cardinality $d$. \\
$\bullet$\quad $P_I^{(t)} 
\ \coloneqq \
\prod_{i\in I} P_i^{(t)}$. \\
$\bullet$\quad $v_d^{(n)} 
\ \coloneqq \
2^{d-m} V_d^{(n)} = \sum_{I} \prod_{i\in I} S_i^{(n)} = \sum_{I} \mathfrak{p}_I^{(n)}$: the $d$-volume without the constant $2^{m-d}$. \\
$\bullet$\quad $\mathfrak{p}_I^{(n)}
\ \coloneqq \
\prod_{i\in I} S_i^{(n)} = \prod_{t = 0}^n P_I^{(t)}$: the product of the sides in $I$. \\
$\bullet$\quad $\mathfrak{m}_d^{(n)}
\ \coloneqq \
\max_{I} \mathfrak{p}_I^{(n)}$: the maximum product of $d$ sides.
\vspace{5pt}

It suffices to show that the random variables $v_d^{(n)}$ converge to strong Benford behavior, because $v_d^{(n)}$ and $V_d^{(n)}$ only differ by a fixed multiplicative constant of $2^{m-d}$. Indeed, if $X$ is Benford, so is $cX$ for any fixed $c > 0$.
Moreover, what we like is to control such a sum of products $\sum_I \mathfrak{p}_I^{(n)}$ over $|I|=d$ by using the observation that the maximum product $\mathfrak{m}_d^{(n)}$ should typically be many orders of magnitude larger than the other products. We quantify this statement in the form of Lemma \ref{lem:WaferLemma}, which is the tool that allows us to control the strong Benford behavior of our sum of random variables, allowing us to if one ascertains that the strong Benfordness of the maximum is suitable. In rare instances, such as $\mathfrak{m}_d^{(n)} = (B-\varepsilon) \cdot B^k$ where $\varepsilon>0$ is small, the Benfordness of $\mathfrak{m}_d^{(n)} = \mathfrak{p}_{\mathcal{I}}^{(n)}$ for some $|\mathcal{I}|=d$ does not translate well to the Benfordness of $\mathfrak{m}_d^{(n)} + \sum_{J\neq\mathcal{I}} \mathfrak{p}_{\mathcal{I}}^{(n)}$, since there is an overflow of the digits base $B$ which tampers with the distribution of the significand greatly. We handle these events, showing they almost always never occur (i.e., with probability tending towards $0$) in a standard way (cf. \S9.3.2 of \cite{MT-B}).

We first require a lemma.

\begin{lem}[Wafer Lemma]
\label{lem:WaferLemma}
Let $0<\delta_n<1$ be a decreasing sequence. Then the probability that $v_d^{(n)}$ is at most $(1+\delta_n)$ times $\mathfrak{m}_d^{(n)}$ is
\begin{equation}
\label{eq:WaferLemma}
\Prob\left(\mathfrak{m}_d^{(n)} \leq v_d^{(n)} \leq (1+\delta_n)\mathfrak{m}_d^{(n)}\right) \ = \
1 - O\left(\frac{-\log\delta_n}{\sqrt{n}}\right) ,
\end{equation}
where the implied constant depends on the distribution of $Y_j^{(t)}$ and $m$. We say that such an event at time $n$ is a $\delta_n$-Wafer.
\end{lem}

\begin{proof}
Our goal is to show that as $n\rightarrow \infty$, it is with probability tending to $1$ that there exists a product $\mathfrak{p}_I^{(n)}$ which is significantly greater in magnitude than the other products $\mathfrak{p}_J^{(n)}$ for $J\neq I$.
That is, it is with probability tending to $1$ that there exists an indexing set $\mathcal{I}$ which has the largest product and is large in the sense that $\log \mathfrak{p}_{\mathcal{I}}^{(n)} - \log \mathfrak{p}_J^{(n)} \geq \alpha_n$ for all $J\neq \mathcal{I}$, where $\alpha_n$ slowly tends towards infinity. We first write for every $I$
\begin{equation}
\label{eq:log p sum of log P and log S}
\log \mathfrak{p}_I^{(n)}
\ = \
\sum_{i\in I} \log S_i^{(n)}
\ = \
\sum_{i\in I} \sum_{t\leq n} \log P_i^{(t)}.
\end{equation}
Notice that, due to the inequality below, we may reduce to the $d=1$ dimensional case, since showing that it tends to $1$ will squeeze all other probabilities. Indeed,
\begin{equation}
\Prob\left(\bigcup_{|\mathcal{\mathcal{I}}| = d} \bigcap_{J\neq \mathcal{I}} \{\log \mathfrak{p}_{\mathcal{I}}^{(n)} - \log \mathfrak{p}_J^{(n)} \geq \alpha_n\}\right)
\ \geq \
\Prob\left(\bigcup_{i=1}^m \bigcap_{j\neq i} \{\log S_i^{(n)} - \log S_j^{(n)} \geq \alpha_n\}\right).
\end{equation}
This can be seen by using the middle expression for $\log\mathfrak{p}_I^{(n)}$ in (\ref{eq:log p sum of log P and log S}).
Notice that for $\alpha_n>0$, the union of events over $i$ is disjoint, therefore we calculate
\begin{align}
\Prob\left(\bigcup_{i = 1}^m \bigcap_{j\neq i} \{\log S_i^{(n)} - \log S_j^{(n)} \geq \alpha_n\}\right) \ &= \
\sum_{i=1}^m \ 
\Prob\left(\bigcap_{j\neq i} \{\log S_i^{(n)} - \log S_j^{(n)} \geq \alpha_n\}\right)
\nonumber \\ &= \ 
\sum_{i=1}^m \
\int_{-\infty}^\infty \, f_i^{(n)}(s) \prod_{j\neq i} F_j^{(n)}(s-\alpha_n) \,ds,
\label{eq:Wafer lemma integral expression}
\end{align}
where we have used the integral version of the law of total probability with respect to the values that the maximum value $s = \log S_i^{(n)}$ may take, as well as independence of the $S_j$'s. The functions $f_j^{(n)},F_j^{(n)}$ denote the PDF and CDF of $\log S_j^{(n)}$ respectively.
One version of the Berry--Esseen theorem (cf. \cite{Berry} and \cite{Esseen}) gives us, in consideration of (\ref{eq:log p sum of log P and log S}) for each $\log S_j^{(n)}$,
\begin{equation}
\label{eq:Berry Esseen log Sj}
F_{j}^{(n)}(x) \ = \
\Phi\left(\frac{x-n\cdot \mu_P}{\sqrt{n}\cdot\sigma_P}\right) + 
O_P\left(\frac{1}{\sqrt{n}}\right)
\end{equation}
where $\Phi$ is the PDF of the standard normal $\mathcal{N}(0,1)$, and the implied constant for $O_P(1/\sqrt{n})$ is uniform over $x\in\R$.
By our convention in Definition \ref{def:LFP}, $\mu_P = \mathbb{E}[\log_B P_j^{(1)}]$ and $\sigma_P = {\rm Var}[\log_B P_j^{(1)}]$ are uniform over $1 \leq j \leq m$.
Applying (\ref{eq:Berry Esseen log Sj}) to (\ref{eq:Wafer lemma integral expression}) yields, for $1 \ll \alpha_n \ll \sqrt{n}$,
\begin{align}
\sum_{i=1}^m \int_{-\infty}^\infty
f_i^{(n)}(s)
&
\left(\Phi\left(\frac{s-n\cdot\mu_P}{\sqrt{n}\cdot\sigma_P}\right) + O_P\left(\frac{\alpha_n}{\sqrt{n}}\right)\right)^{m-1}\,ds
\nonumber \\ &\qquad\qquad = \
\left(\sum_{i=1}^m
\int_{-\infty}^\infty f_i^{(n)}(s) \Phi\left(\frac{s-n\cdot\mu_P}{\sqrt{n}\cdot\sigma_P}\right)^{m-1} \,ds\right) \ + \  
O_{P,m}\left(\frac{\alpha_n}{\sqrt{n}}\right).
\end{align}
Integrating by parts, applying (\ref{eq:Berry Esseen log Sj}) to $F_i^{(n)}$ and absorbing error, we obtain
\begin{equation}
O_{P,m}\left(\frac{\alpha_n}{\sqrt{n}}\right) + \sum_{j=1}^m
\left(
1 - 
\int_{-\infty}^{\infty}
\frac{m-1}{\sqrt{n}\cdot\sigma_P}\cdot
\Phi\left(\frac{s- n\cdot\mu_P}{\sqrt{n}\cdot\sigma_P}\right)^{m-1}
\Phi^{\prime}\left(\frac{s- n\cdot\mu_P}{\sqrt{n}\cdot\sigma_P}\right)\,ds\right).
\end{equation}
One may recognize that the above integrand has primitive $(1-\frac{1}{m})\Phi(\frac{s-n\cdot\mu_P}{\sqrt{n}\cdot\sigma_P})^m$, and so each integral contributes $1-\frac{1}{m}$, leaving us with
\begin{equation}
1 - O_{P,m}\left(\frac{\alpha_n}{\sqrt{n}}\right).
\end{equation}
Taking $\alpha_n = -\log(\delta_n/\binom{m}{d})$, we have by considering subevents
\begin{align}
\Prob\left(\mathfrak{m}_d^{(n)} \leq v_d^{(n)} \leq (1+\delta_n)\mathfrak{m}_d^{(n)}\right)
\ &\geq \ 
\Prob\left(\bigcup_{|\mathcal{I}|=d} \bigcap_{J\neq \mathcal{I}} \{\log\mathfrak{p}_{\mathcal{I}}^{(n)} - \log\mathfrak{p}_J^{(n)} \geq -\log(\delta_n/\tbinom{m}{d})\}\right)
\nonumber \\ &\geq \ 
\Prob\left(\bigcup_{i = 1}^m \bigcap_{j\neq i} \{\log S_i^{(n)} - \log S_j^{(n)} \geq -\log(\delta_n/\tbinom{m}{d})\}\right)
\nonumber \\ &= \ 
1 - O_{P,m}\left(\frac{-\log\delta_n}{\sqrt{n}}\right).
\end{align}
This finishes our proof.
\end{proof}
We claim that Lemma \ref{lem:WaferLemma} reduces the question of strong Benford behavior of $v_d^{(n)}$ to $\mathfrak{m}_d^{(n)}$. That is, the Wafer lemma implies

\begin{lem}[Reduction to Max]
Assume $\mathfrak{m}_d^{(n)}$ converges to strong Benford behavior. Then $v_d^{(n)}$ does as well.
\end{lem}

\begin{proof}
Let $E_n$ be the event that $\mathfrak{m}_d^{(n)}$ and $v_d^{(n)}$ are a $\delta_n$-Wafer and $(1+\delta_n)S_B(\mathfrak{m}_d^{(n)}) < B$. We condition on this event to prevent an overflow of the order of magnitude. Then
\begin{equation}
S_B(\mathfrak{m}_d^{(n)})
\ \leq \
S_B(v_d^{(n)})
\ \leq \
(1+\delta_n)S_B(\mathfrak{m}_d^{(n)}).
\end{equation}
Moreover, the conditional probabilities are
\begin{equation}
\Prob(S_B(\mathfrak{m}_d^{(n)}) \leq D/(1+\delta_n) \mid E_n)
\ \leq \ 
\Prob(S_B(v_d^{(n)}) \leq D \mid E_n)
\ \leq \
\Prob(S_B(\mathfrak{m}_d^{(n)}) \leq D \mid E_n).
\end{equation}
Making basic estimates such as inclusion-exclusion, we estimate the unconditional probability as
\begin{equation}
\label{eq:uses inclusion exclusion}
\Prob\Big(S_B(\mathfrak{m}_d^{(n)})\leq D/(1+\delta_n)\Big) + \Prob(E_n) - 1
\ \leq \ 
\Prob(S_B(v_d^{(n)})\leq D)
\ \leq \ 
\Prob\Big(S_B(\mathfrak{m}_d^{(n)}) \leq D\Big) + 1 - \Prob(E_n).
\end{equation}
We show that $v_d^{(n)}$ converges to strong Benford behavior by taking $\delta_n\to0$ at a slow enough rate.

Because $\delta_n\to0$, one has
\begin{equation}
\Prob\Big((1+\delta_n)S_B(\mathfrak{m}_d^{(n)}) < B\Big) \to 1.
\end{equation}
Also, by assuming that $\delta_n$ slowly goes to zero in the sense that $\log(\frac{1}{\delta_n}) = o(\sqrt{n})$, we have by Lemma \ref{lem:WaferLemma} that
\begin{equation}
\Prob(\delta_n\text{-wafer}) \ = \ 
1 - O\left(\frac{-\log\delta_n}{\sqrt{n}}\right) \to 1.
\end{equation}
Since $E_n$ is the intersection of these two events, we see that $\Prob(E_n)\to 1$ because of an inclusion-exclusion bound that we also used to obtain (\ref{eq:uses inclusion exclusion}).
\begin{equation}
\Prob(A\cap B) \geq \Prob(A) + \Prob(B) - 1.
\end{equation}
By our assumption that $\mathfrak{m}_d^{(n)}$ converges to strong Benford behavior, we have that 
\begin{equation}
\Prob\Big(S_B(\mathfrak{m}_d^{(n)})\leq D\Big) \to \log_B(D).
\end{equation} Because of uniform convergence (due to compactness), we also have that 
\begin{equation}
\Prob\Big(S_B(\mathfrak{m}_d^{(n)})\leq D/(1+\delta_n)\Big) \to \log_B(D).
\end{equation}
Therefore by the squeeze theorem, we deduce that
\begin{equation}
\Prob\Big(S_B(v_d^{(n)})\leq D\Big) \to \log_B(D) ,
\end{equation}
provided that $\mathfrak{m}_d^{(n)}$ converges to strong Benford behavior.
\end{proof}

This proves Theorem $\ref{thm:Redux2Max}$, because $v_d^{(n)}$ and $V_d^{(n)}$ differ by only a constant multiplicative factor of $2^{m-d}$.

\section{A Family of Distributions whose Maximum Side-Lengths are Benford}
\label{sec:ExampleFamily}

For the sake of clean and transparent analysis, we select as our example family identically log-uniform distributions: $\log_B P_i \sim {\rm Uniform}(a,b)$ where $a<b \leq 0$.
By shifting and scaling each logarithm of a proportion by a constant, we realize that we may ``work'' with the normalized distribution ${\rm Uniform}(-\sqrt{3},\sqrt{3})$, which has mean zero and variance one. Of course, this means that we are no longer strictly considering a physically realistic linear fragmentation process, because the boxes $\mathfrak{B}_n$ no longer form a descending chain, however for the sake of purely analyzing the Benfordness of our system, this statistical normalization clearly generalizes, and we lose nothing by assuming it. For $1 \leq i \leq m$, we let
\begin{equation}
Z_i^{(n)} \ \coloneqq \ \frac{\log_B (P_i^{(1)} \dotsi P_i^{(n)}) - n\mu_P}{\sqrt{n}\cdot \sigma_P} \ = \
\frac{\log_B (P_i^{(1)} \dotsi P_i^{(n)})}{\sqrt{n}}.
\end{equation}
If $f_n(x)$ denotes the probability density function of any one of the random variables above, then its characteristic function is
\begin{equation}
\widehat{f_n}(k) \ \coloneqq \ \int_{-\infty}^\infty f_n(x) e^{ikx}\,dx \ = \ 
\left({\rm sinc}\left(\frac{k\sqrt{3}}{\sqrt{n}}\right)\right)^n.
\end{equation}
Note that the function ${\rm sinc}:\R\to\R$ is defined by
\begin{equation}
{\rm sinc}(u) \ \coloneqq \
\begin{cases}
\frac{\sin(u)}{u} & u \neq 0 \\
1 & u = 0
\end{cases}.
\end{equation}
The above formula for $\widehat{f_n}(k)$ follows by writing $f_n(x)$ as the $n$-fold convolution of the PDF of $P_i$, and then normalizing the PDF by subtracting mean and dividing by variance. We then use the fact that the characteristic function (or Fourier transform) turns convolutions into products.

Our goal is to produce an estimate of the closeness of the PDF $f_n(x)$ and the Gaussian function
\begin{equation}
\varphi(x) \ \coloneqq \ \frac{1}{\sqrt{2\pi}} e^{-x^2/2}.
\end{equation}
As we will see, this closeness will allow us to correctly estimate the probability of events involving the significand.

\begin{lem}
\label{lem:GoodEstimate}
The following estimate is satisfied by the random variables $Z_i^{(n)}$.
\begin{equation}
\label{eq:GoodEstimate}
f_n(x) \ = \ 
\varphi(x) + O(n^{-1+4\varepsilon}).
\end{equation}
\end{lem}

\begin{proof}
We begin by writing, for $|k| \leq n^{\varepsilon}$,
\begin{align}
\widehat{f_n}(k) \ &= \ 
\left(1 - \frac{k^2}{2n} + O\left(\frac{k^4}{n^2} \right)\right)^n
\nonumber \\ &= \ 
\exp \left[ n \log \left(1 - \frac{k^2}{2n} + O\left(\frac{k^4}{n^2}\right)\right) \right]
\nonumber \\ &= \ 
\exp \left[ n \left(\log \left(1 - \frac{k^2}{2n}\right) + \frac{1}{1-\frac{k^2}{2n}} O\left(\frac{k^4}{n^2}\right)\right)\right]
\nonumber \\ &= \ 
\exp \left[ n \left(-\frac{k^2}{2n} + O\left(\frac{k^4}{n^2}\right) + O(1) O\left(\frac{k^4}{n^2}\right)\right)\right]
\nonumber \\ &= \ 
\exp\left[-\frac{k^2}{2n} + O(n^{-1+4\varepsilon})\right]
\nonumber \\ &= \ 
e^{-k^2/2} \cdot \left(1 + O(n^{-1+4\varepsilon})\right).
\end{align}
In terms of Fourier inversion, this allows us to manage the bulk part of our sum, namely we obtain
\begin{equation}
\frac{1}{2\pi} \int_{|k|\leq n^{\varepsilon}} \widehat{f_n}(k) e^{-ikx}\,dk \ = \ \varphi(x) + O(n^{-1+4\varepsilon}).
\end{equation}
It therefore suffices to show that we have adequate bandwidth $|k|\leq n^\varepsilon$ for recovering $f_n(x)$ from $\widehat{f_n}(k)$ as $n\to\infty$, i.e., to bound the strength of higher frequencies. Using the bound $|{\rm sinc}(u)| \leq |u|^{-1}$, we obtain for all $k$ that
\begin{equation}
|\widehat{f_n}(k)|
\ \leq \ 
\left(\frac{|k|\sqrt{3}}{\sqrt{n}}\right)^{-n}.
\end{equation}
This immediately shows that we may discard the set of frequencies $|k| \geq n^{1/2}$, since our bound yields
\begin{equation}
\int_{|k|\geq n^{1/2}} |\widehat{f_n}(k)| \,dk 
\ \leq \ 
\frac{2}{n-1} \cdot 3^{-n/2}
\ = \ 
O\left(\frac{3^{-n/2}}{n}\right).
\end{equation}

We want to show that the Fourier inversion over the middle range of frequencies $n^{\varepsilon} \leq |k| \leq n^{1/2}$ is also a small error term. This is because for $n^{\varepsilon} \leq |k| \leq n^{1/2}$, we estimate, using ${\rm sinc}(\sqrt{3}u) \leq 1 - \frac{u^2}{2.1}$ for small $u$, as well as $(1-\frac{1}{N})^N \leq e^{-1}$ for $N$ large,
\begin{equation}
0
\ \leq \ 
\widehat{f_n}(k) 
\ \leq \  \widehat{f_n}(n^{-\varepsilon})
\ \leq \ 
\left(1 - \frac{n^{-1+2\varepsilon}}{2.1}\right)^n
\ \leq \ 
e^{-n^{2\varepsilon}/2.1}.
\end{equation}
This allows us to estimate
\begin{equation}
\int_{n^\varepsilon \leq |k|
\ \leq \ 
n^{1/2}} |\widehat{f_n}(k)| \,dk \leq 2 n^{1/2} e^{-n^{2\varepsilon}/2.1}
\ = \ 
O(n^{1/2} e^{-n^{2\varepsilon}/2.1}).
\end{equation}
We therefore have proven, combining all of our estimates, that
\begin{equation}
f_n(x) \ = \ \frac{1}{2\pi}\int_{-\infty}^\infty \, \widehat{f_n}(k) e^{-ikx} \,dk = \varphi(x) + O(n^{-1+4\varepsilon}) + O(n^{1/2}e^{-n^{2\varepsilon}/2.1}) + O(3^{-n/2}/n).
\end{equation}
This yields the desired estimate.

\end{proof}

Using estimate (\ref{eq:GoodEstimate}), we are able to prove Theorem \ref{thm:SpecialCaseMaxSideLength}.

\begin{rek}
We crucially rely on the fact that the $Z_i^{(n)}$ are independent. For $d > 1$, this is no longer true, as $\mathfrak{p}_I^{(n)}, \mathfrak{p}_J^{(n)}$ share proportion cuts even if $I \neq J$. This obstruction should be able to be removed with further work.
\end{rek}

Corollary \ref{cor:BenfordPerimeter} follows from applying Theorem \ref{thm:Redux2Max} and Theorem \ref{thm:SpecialCaseMaxSideLength}. Let us now prove Theorem \ref{thm:SpecialCaseMaxSideLength}.

\begin{proof}
Because the side-lengths $S_i^{(n)} = P_i^{(1)} \dotsi P_i^{(n)}$ are independent and identically distributed, the probability density function $g_n(x)$ for the normalized random variable $(\log_B\mathfrak{m}_1^{(n)} - n\mu_P)/\sqrt{n}\sigma_P$ is given by
\begin{equation}
\label{eq:NormalizedMaxPDF}
g_n(x) \ \coloneqq \ m F_n(x)^{m-1} f_n(x).
\end{equation}
This is a basic fact about order statistics (see \cite{Mil2}).
Notice that the support of $Z_i^{(n)}$ is $[-\sqrt{3n},\sqrt{3n}]$.
Using Lemma \ref{lem:GoodEstimate}, we have
\begin{equation}
\label{eq:GoodCDF}
F_n(x) \ = \ 
\int_{-\sqrt{3n}}^x \, \varphi(x) + O(n^{-1+4\varepsilon}) \,dx \ = \ \Phi(x) + O(n^{-1/2+4\varepsilon}).
\end{equation}
This allows us to say that the maximum of approximate Gaussian random variables is approximately the maximum of Gaussian random variables, i.e., expanding (\ref{eq:NormalizedMaxPDF}) using (\ref{eq:GoodCDF}) and (\ref{eq:GoodEstimate}), one derives
\begin{equation}
g_n(x) \ = \ m\Phi(x)^{m-1}\varphi(x) + O(m^2n^{-1/2+4\varepsilon}\varphi(x)) + O(m^2n^{-3/2+8\varepsilon}) + O(mn^{-1+4\varepsilon}).
\end{equation}
Our last step is to compute the probability that $\log_B \mathfrak{m}_1^{(n)} \in (a,b) + \Z$ where $(a,b) \subset (0,1)$. This is given by integrating $g_n(x)$ over the set $E_n = (a/\sqrt{n},b/\sqrt{n}) + \mathbb{Z}/\sqrt{n}$. Thus the probability is
\begin{equation}
\int_{E_n} \, g_n(x)\,dx \ = \ 
\int_{E_n} \, m\Phi(x)^{m-1}\varphi(x) \,dx +
O(m^2n^{-1/2+4\varepsilon}) + O(m^2n^{-1+8\varepsilon}) + O(mn^{-1/2+4\varepsilon}).
\end{equation}
The integral on the right hand side represents the probability that the max of $m$ Gaussian random variables lies in the set $E_n$, and the probability approaches $(b-a)$. Indeed, one way to see this is that we are performing an improper Riemann sum of width $1/\sqrt{n}$ on the fixed Riemann-integrable function $m\Phi(x)^{m-1}\varphi(x)$, and that the set $E_n$ simply is a ``dense'' subset of the rectangles. Therefore the integral over $E_n$ in the limit approaches the ``probability'' that a chosen rectangle intersects $E_n$, which is $(b-a)$, times the limit of the improper Riemann sums of $m\Phi(x)^{m-1}\varphi(x)$, which is simply $1$. Therefore choosing $\varepsilon < 1/8$, we have
\begin{equation}
\lim_{n\to\infty} \int_{E_n} g_n(x) \,dx = (b-a).
\end{equation}
From this we deduce that the maximum perimeter sequence $\mathfrak{m}_1^{(n)}$ exhibits Strong Benford behavior.
\end{proof}

\section{Future Work} \label{sec:conclusionandfuturework}

We conjecture that the maximum criterion, i.e., the assumption in Theorem \ref{thm:Redux2Max}, holds for a large family of proportion cut distributions. More precisely, we conjecture the following.

\begin{conj}
\label{conj:MaxCriterionHolds}
Every linear-fragmentation process (that is continuous, with finite mean, variance, and third moment) satisfies the maximum criterion in all dimensions $1 \leq d \leq m$.
\end{conj}

From our work, this conjecture implies the following corollary.
\begin{cor}[Strong Benfordness]
\label{cor:LFP is StrongBenford}
Assume that Conjecture \ref{conj:MaxCriterionHolds} holds. Then every linear-fragmentation process satisfies the strong form of Benford's law for all dimensions $1 \leq d \leq m$.
\end{cor}

While the tools we have employed thus far in our work with linear-fragmentation processes are distinct from the methods used previously in working with branching-fragmentation processes, the only substantial difference between linear-fragmentation and branching-fragmentation is the presence of a binary tree of weakly correlated events. Applying linearity of expectation, one sees that the expectation values of the leaves of the tree are the expectation value of the end of a linear-fragmentation process with the same height. Therefore we obtain the following corollary.

\begin{cor}
Assume that Conjecture \ref{conj:MaxCriterionHolds} holds. Then every branching-fragmentation process (with proportion cuts $P_i$ as in Definition \ref{def:LFP}) satisfies
\begin{equation}
\lim_{n\to\infty}\mathbb{E}[\rho_d^{(n)}(s)] \ = \ 
\log_B(s),
\end{equation}
where
\begin{equation}
\rho_d^{(n)}(s) \ = \ \frac{1}{2^{mn}} \sum_{i=1}^{2^{mn}}\varphi_s(\Vol_d(\mathfrak{B}_{i})).
\end{equation}
\end{cor}

\subsection{Restating the Mellin Condition in Terms of Characteristic Functions}

Our last remark section for this paper concerns how to interpret the Mellin condition that is presented in the works of \cite{Beck} and \cite{DM} in terms of characteristic functions for the logarithm of proportion cuts.

If $f(t)$ is the probability density of $P\in(0,1)$ where $P$ is a proportion cut, then we want to state the Mellin condition
\begin{equation}
\lim_{n\to\infty} \sum_{\substack{\ell = -\infty \\ \ell \neq 0}}^\infty \left| \mathcal{M}[f]\left(1+\frac{i \ell}{\log B}\right) \right|^n \ = \ 0,
\end{equation}
in terms of a characteristic function condition for $f_n(x)$, the probability density of $Z^{(n)}$, defined as
\begin{equation}
Z^{(n)} \ \coloneqq \ \frac{\log_B(P^{(1)} \dotsi P^{(n)}) - n\mu_P}{\sqrt{n}\cdot \sigma_P},
\end{equation}
where $\mu_P,\sigma_P$ are constants which denote the mean and variance of $\log_B P$.

\begin{rek}
We have slightly modified the Mellin condition originally specified in equation \ref{eq:MellinCondition}. The reason for this change is that we want (i) an arbitrary base $B$, (ii) an answer in terms of characteristic functions rather than Fourier transforms (dropping $2\pi$), and (iii) we may without loss of generality always take $f_u = f$ (see Remark 1.7 of \cite{DM}).
\end{rek}

Now take $t = P$, $s = \log_B P$, and define
\begin{equation}
g(s) = f(B^s) \cdot B^{s} \ln(B).
\end{equation}
A change of variables from $t$ to $s$ gives 
\begin{equation}
\mathcal{M}[f]\left(1+\frac{i \ell}{\log B}\right) \ = \ 
\int_0^\infty f(t) t^{i\ell/\ln(B)} \,dt \ = \ 
\int_{-\infty}^\infty g(s) e^{is\ell}\,ds
\ = \ 
\widehat{g}(\ell)
\end{equation}

We have by definition of the random variable $Z^{(n)}$ that $f_n(x) = \sqrt{n}\cdot \sigma_P \cdot (g^{*n})(n\mu_P + \sqrt{n}\cdot \sigma_P x)$. Without loss of generality assume that $\mu_P = 0$ and $\sigma_P = 1$. Then one obtains by applying the characteristic transform
\begin{equation}
\widehat{f_n}(k) = (\widehat{g})^n(k/\sqrt{n}).
\end{equation}
Therefore we see that the Mellin condition is equivalent to the statement that
\begin{equation}
\lim_{n\to\infty}
\sum_{\substack{\ell = -\infty \\ \ell \neq 0}}^\infty |\widehat{f_n}(\ell\sqrt{n})|
\ = \ 0.
\end{equation}
Thus we observe that this is a very mild regularity condition, since we expect $\widehat{f_n}(k) \approx e^{-k^2/2}$ for nice $P$. We have seen this condition concretely hold for the family of proportion distributions in Section \ref{sec:ExampleFamily}.







\newpage

\end{document}